\documentclass[12]{amsart}
\newtheorem{thm}{Theorem}

\newtheorem{cor}[thm]{Corollary}
\newtheorem{lem}[thm]{Lemma}

\newtheorem{prop}[thm]{Proposition}
\newtheorem{exa}[thm]{Example}

\theoremstyle{definition}
\newtheorem{defn}[thm]{Definition}

\begin{document}
\noindent ({\it J. Algebra, to appear})
\title[Strongly clean matrices]{Strong cleanness of
the $2\times 2$ matrix ring over a general local ring}
\author[Yang and Zhou]{Xiande Yang and Yiqiang Zhou}

\address{Department of Mathematics and Statistics, Memorial University of
Newfoundland, St.John's, Newfoundland A1C 5S7, Canada\\
and Department of Mathematics, Harbin Institute of Technology, Harbin 150001, China}
\email{xiande.yang@gmail.com}
\address{Department of Mathematics and Statistics, Memorial University of
Newfoundland, St.John's, Newfoundland A1C 5S7, Canada}
\email{zhou@math.mun.ca}

\thanks{Corresponding author: Yiqiang Zhou, Department of
Mathematics and Statistics, Memorial University of Newfoundland, St.John's,
NL A1C 5S7, Canada,
Tel: 1-709-737-8797, Fax: 1-709-737-3010}
\date{Submitted in November 2006; revised in June 2007 and in February 2008}
\maketitle
\begin{abstract}A ring $R$ is called strongly clean if every element of $R$ is the sum of a
unit and an idempotent that commute with each other. A recent result of
Borooah, Diesl and Dorsey \cite{BDD05a} completely characterized the commutative local rings
$R$ for which ${\mathbb M}_n(R)$ is strongly clean.  For a general local ring $R$ and $n>1$,
however, it is unknown when the matrix ring ${\mathbb M}_n(R)$ is strongly clean.
Here we completely determine the local rings $R$ for which ${\mathbb M}_2(R)$ is strongly clean.

\end{abstract}

\bigskip
\section{Introduction}
In this paper, all rings are associative with unity and all modules are unitary.
For a module $M$ over a ring $R$, the $R$-homomorphisms of $M$ are written on the opposite
side of their arguments, and the ring of endomorphisms of $M$ is denoted by ${\rm End}(M_R)$ or
${\rm End}(_RM)$. We begin
by recalling a well-known notion in ring theory. An element $a$ in a ring $R$ is called {\it strongly
$\pi$-regular} if both chains $aR\supseteq a^2R\supseteq \cdots$ and $Ra\supseteq Ra^2\supseteq
\cdots$ terminate, and the ring $R$ is called {\it strongly $\pi$-regular} if every element of $R$ is
strongly $\pi$-regular (or equivalently, the chain $aR\supseteq a^2R\supseteq \cdots$ terminates
for all $a\in R$, by Dischinger \cite{D76}). Thus, one-sided perfect rings are strongly
$\pi$-regular. A result of Armendariz,
Fisher and Snider \cite{AFS78} says that for a module $M_R$, $\varphi\in {\rm End}(M_R)$ is
strongly $\pi$-regular iff
$M={\rm ker}(\varphi^n)\oplus {\rm Im}(\varphi^n)$ for some $n\ge 1$ (i.e., $\varphi$ is a so called
Fitting endomorphism). The notion of a strongly clean ring was introduced by Nicholson \cite{N99}
in 1999. An element $a$ of a ring $R$ is called {\it strongly clean} if $a=e+u$ where $e^2=e\in R$ and $u$
is a unit of $R$ with $eu=ue$, and the ring $R$ is called {\it strongly clean} if each of its elements is
strongly clean. Clearly, local rings are strongly clean. By a result of Burgess and Menal
\cite{BM88}, every strongly $\pi$-regular ring is strongly clean. In \cite{N99}, Nicholson gave a
direct proof of the result that every strongly $\pi$-regular element of a ring is strongly clean,
and furthermore he offered the interesting viewpoint that strongly clean elements are natural
generalizations of the strongly
$\pi$-regular elements by establishing the following results: for $\varphi\in {\rm End}(M_R)$,
$\varphi$ is strongly $\pi$-regular iff there exists a decomposition $M=P\oplus Q$ such that
$\varphi:P\rightarrow P$ is an isomorphism and $\varphi: Q\rightarrow Q$ is nilpotent; and $\varphi$
is strongly clean iff there exists a decomposition $M=P\oplus Q$
such that $\varphi: P\rightarrow P$ and $1-\varphi: Q\rightarrow Q$ are isomorphisms.

\medskip
In considering whether the class of strongly clean rings is Morita invariant,
Nicholson \cite{N99} raised two questions: if $R$ is strongly clean with $e^2=e\in R$,
is $eRe$ strongly clean? is ${\mathbb M}_n(R)$ strongly clean? In her 2002
unpublished manuscript \cite{S02}, S\'{a}nchez Campos answered the first question affirmatively
and gave a counter-example to the second question. In 2004, Wang and Chen \cite{WC04},
independently, published a counter-example to the second question. Surprisingly, the authors
of the two articles came to the same counter-example ${\mathbb Z}_{(2)}$, the localization of
$\mathbb Z$ at the prime ideal $(2)$. This motivated the  authors of \cite{BDD05a, CYZ05, CYZ06}
to
consider the question: when is ${\mathbb M}_n(R)$ strongly clean? Observing a pattern of the
$2\times 2$ idempotent matrices over a commutative local ring, using techniques from linear algebra
the authors of \cite{CYZ05, CYZ06} characterized the commutative local rings $R$ for which
${\mathbb M}_2(R)$ is strongly clean. The authors of \cite{BDD05a} had a different approach to
this question. Using Nicholson's decomposition theorem, and considering  different types of
factorization in $R[t]$, for each $n$ they characterized the commutative local rings $R$ for
which ${\mathbb M}_n(R)$ is strongly clean. Thus, the above question is completely settled when
$R$ is a commutative local ring.

\medskip
In this paper, we continue the study of this question, focusing on
the question of when ${\mathbb M}_2(R)$ is strongly clean, for
noncommutative local rings $R$. Following P.M. Cohn
\cite[p.17]{C85}, a ring $R$ is called {\it projective-free} if
every finitely generated projective $R$-module is free of unique
rank. In Section 2, using the aforementioned decomposition theorem
of Nicholson found in \cite{N99}, we prove that, for a
projective-free ring $R$,  all `non trivial' strongly clean matrices
of ${\mathbb M}_n(R)$ are similar to a certain type of block
diagonal matrix. For a local ring $R$ with $n=2$, this simply says
that $A\in {\mathbb M}_2(R)$ is strongly clean iff either $A$ is
invertible or $I-A$ is invertible or $A$ is similar to
 $\begin{pmatrix}t_0&0\\
                       0&t_1\end{pmatrix}$, where $1-t_0, t_1\in J(R)$. This result is put to use
when we establish some (easily verifiable) criteria for a $2\times 2$ matrix ring over a local ring
to be strongly clean, and, as consequences, new families of strongly clean rings are presented in
Section 3. It is noticed that the same idea can be used to characterize the local rings $R$ for
which ${\mathbb M}_2(R)$ is strongly $\pi$-regular, and this discussion is recorded in Section 4.

\medskip
As usual, ${\mathbb Z}$ denotes the ring of integers. The polynomial ring over a ring $R$ in the
indeterminate $t$ is denoted by $R[t]$. For an endomorphism $\sigma$ of a ring $R$ with
$\sigma (1)=1$, let $R[[x,\sigma]]$ denote the ring of left skew power series over $R$. Thus,
elements of $R[[x,\sigma]]$ are power series in $x$ with coefficients in $R$ written on the left,
subject to the relation $xr=\sigma(r)x$ for all $r\in R$. The Jacobson radical and the group of
units of a ring $R$ are denoted by $J(R)$ and $U(R)$ respectively. For an integer $n>0$, we
write ${\mathbb M}_n(R)$ for the $n\times n$ matrix ring over $R$ whose identity element we write as $I_n$ or $I$,
and ${\rm GL}_n(R)$ for the group of all invertible $n\times n$ matrices over $R$.

\section{Strongly clean matrix rings}

In this section, we obtain some necessary and sufficient conditions for a $2\times 2$ matrix ring
over a local ring to be strongly clean, which  will be used to give new families of strongly clean
rings in the next section.

\medskip
Let $F$ be a free $R$-module with a basis
$\{v_1,v_2,\ldots, v_n\}$ and let $\varphi \in {\rm End}(F_R)$. Then for each $1\le j\le n$,
\begin{equation*}
\varphi (v_j)=\sum_{i=1}^{n}v_ia_{ij}
\end{equation*}
for some $a_{ij}\in R$. Write $M_{\varphi}=(a_{ij})\in {\mathbb M}_n(R)$. It is well known that
the map  ${\rm End}(F_R)\rightarrow {\mathbb M}_n(R)$, given by $\varphi\mapsto M_{\varphi}$, is a
ring isomorphism.
Moreover, changing the basis of $F_R$ yields conjugate matrices (via a change of basis matrix).

\medskip
We need Nicholson's characterization of strongly clean matrices, which is a transition from a result of
his we are recalling.
\begin{lem}\label{lem:1}\cite[Theorem 3]{N99}
Let $M_R$ be a module. The following are equivalent for $\varphi\in {\rm End}(M_R)$:
\begin{enumerate}
\item $\varphi$ is strongly clean in ${\rm End}(M_R)$.
\item There is a decomposition $M=P\oplus Q$ where $P$ and $Q$ are $\varphi$-invariant, and
$\varphi |_{P}$ and $(1-\varphi)|_Q$ are isomorphisms.
\end{enumerate}
\end{lem}

A unit $a$ of a ring $R$ is strongly clean because $a=0+a$. If $1-a$ is a unit of $R$, then $a$ is
also strongly clean because $a=1+(a-1)$. A strongly clean element $a\in R$ is called a {\it trivial
strongly clean element} if $a$ is a unit or $1-a$ is a unit, and is called {\it non-trivial}
otherwise.

\medskip
The next lemma is a translation of Nicholson's decomposition in Lemma \ref{lem:1} to matrices,
but
this translation is a useful tool for this paper. The hypothesis here is based on following the
approach of \cite{BDD05a}, and this observation is implicitly used there.

\begin{lem}\label{lem:2}
Let $R$ be a projective-free ring. Then
$A\in {\mathbb M}_n(R)$ is a non-trivial strongly clean matrix
iff $A$ is similar to $\begin{pmatrix}T_0&0\\
                       0&T_1\end{pmatrix}$, where $T_0$ and $I-T_1$ are both invertible and neither
$I-T_0$ nor $T_1$ is invertible.
\end{lem}
\begin{proof} Make the obvious choice of basis, and write down the (block diagonal) matrix
with respect to this basis.
\end{proof}

A local ring is projective-free (see \cite[Corollary 5.5, p.22]{C85}), and
this is why commutative local rings are a natural place to start looking at strongly clean matrix
rings, and why the approach of \cite{BDD05a} works.

\medskip
For $2\times 2$ matrices over a local ring $R$, it is clear that  $A\in {\mathbb M}_2(R)$ is a non-trivial strongly clean matrix iff $A$ is similar to $\begin{pmatrix}t_0&0\\
                       0&t_1\end{pmatrix}$, where $1-t_0\in J(R)$ and $t_1\in J(R)$.

\medskip
Another class of projective-free rings are the (commutative) principal ideal domains.  The claim of the next example follows by Lemma \ref{lem:2}.

\begin{exa}\label{exa:6}
$A\in {\mathbb M}_2({\mathbb Z})$ is a non-trivial strongly clean matrix iff $A$ is similar to one
of the elements in  $\Bigl\{\begin{pmatrix}1&0\\
                       0&0\end{pmatrix}, \begin{pmatrix}-1&0\\
                       0&0\end{pmatrix}, \begin{pmatrix}1&0\\
                       0&2\end{pmatrix}$ and $\begin{pmatrix}-1&0\\
                       0&2\end{pmatrix}\Bigr\}$.
\end{exa}

One of the primary things that makes $2\times 2$ matrix rings over local rings easier to deal with than general matrix rings is that all matrices which are neither a unit nor $I$ minus a unit are similar to companion matrices, as the next lemma demonstrates.

\begin{lem}\label{lem:8}Let $R$ be a local ring and let $A\in {\mathbb M}_2(R)$.
Then either $A$ is invertible or $I-A$ is invertible or $A$ is similar to
$\begin{pmatrix}0&w_0\\
                       1&1+w_1\end{pmatrix}$ where $w_0,w_1\in J(R)$.
\end{lem}
\begin{proof}
Let $\varphi\in {\rm End}(R^2_R)$, where neither $\varphi $ nor $1-\varphi$ is invertible. We show
that
there exists a cyclic basis $\{x, \varphi(x)\}$ of $R^2_R$;
with respect to this basis, $\varphi$ corresponds to
$\begin{pmatrix}0&w_0\\
                       1&1+w_1\end{pmatrix}$ for some $w_0,w_1\in J(R)$. Write $\overline R=R/J(R)$
and $\bar r=r+J(R)$ (for $r\in R$), and note that each $\phi \in
{\rm End}(R^2_R)$ induces an endomorphism, denoted $\overline \phi$,
in ${\rm End}({\overline R}^2_{\overline R})$. Therefore, neither
$\overline {\varphi}$ nor $1-\overline{\varphi}$ is invertible in
${\rm End}({\overline R}^2_{\overline R})$, since units lift modulo
the radical. Thus, as vector spaces over $\overline R$, ${\rm
ker}(\overline \varphi)\not= 0$ and ${\rm ker}(1-\overline
\varphi)\not= 0$, and so ${\overline R}^2_{\overline R}={\rm
ker}(\overline \varphi)\oplus {\rm ker}(1-\overline \varphi)$. Take
$0\not= v\in {\rm ker}(\overline \varphi)$ and $0\not= w\in {\rm
ker}(1-\overline \varphi)$. Then $\{v, w\}$ is a basis for
${\overline R}^2_{\overline R}$. Now, lift $v$ and $w$ to $R^2$
(keeping the same names), and let $x=v+w\in R^2$. Then
$\varphi(x)=\varphi(v)+\varphi(w)$, which modulo $JR^2$ equals $w$.
In particular, modulo $JR^2$, $x$ and $\varphi(x)$ are $v+w$ and
$w$, which are a basis for ${\overline R}^2_{\overline R}$, so $x$
and $\varphi(x)$ span $R^2_R$ by Nakayama's Lemma. Moreover, $\{x,
\varphi(x)\}$ is a basis for $R^2_R$ since every local ring is
stably finite. Write $\varphi^2(x)=xa+ \varphi(x)b$. Reducing modulo
$JR^2$, this equation becomes $w=(v+w)\bar a+w\bar b$. Since $\{v+w,
w\}$ is linearly independent in ${\overline R}^{2}_{\overline R}$,
we conclude that $\bar a=0$ and $\bar b=1$ (since $w=(v+w)\cdot
0+w$). This is, $a\in J(R)$ and $b\in 1+J(R)$. The matrix
representation of $\varphi$ with respect to the basis $\{x,
\varphi(x)\}$
is $\begin{pmatrix}0&a\\
                       1&b\end{pmatrix}$,  with $a\in J(R)$ and $b\in 1+J(R)$, as desired.
\end{proof}

For a monic polynomial $h(t)=t^{n}+a_{n-1}t^{n-1}+\cdots+a_{1}t+a_{0}\in R[t]$, the
$n\times n$ matrix $C_{h}= \left[
\begin{array}{cccccc}
0&0&0&\cdots&0&-a_{0}\\
1&0&0&\cdots&0&-a_{1}\\
0&1&0&\cdots&0&-a_{2}\\
\vdots&\vdots&\vdots&\vdots&\vdots&\vdots\\
0&0&0&\cdots&0&-a_{n-2}\\
0&0&0&\cdots&1&-a_{n-1}\\
\end{array}
\right]$ is called the {\it companion matrix} of $h(t)$. A square matrix $A$ over $R$ is called a
companion matrix
if $A=C_{h}$ for a monic polynomial $h(t)$ over $R$. Here is one observation that is true for
companion matrices. Lemma 5 below and its proof were introduced to the authors by the referee in order to
give a conceptual proof of Lemma 6.

\begin{lem}\label{lem:companion matrix}
If  $h(t)=t^{n}+a_{n-1}t^{n-1}+\cdots+a_{1}t+a_{0}\in R[t]$, then $C_{h}^{n}+C_{h}^{n-1}a_{n-1}+
\cdots+C_ha_1+Ia_0=0$ as matrices. {\rm (}That is $C_{h}^{n}+C_{h}^{n-1}(a_{n-1}I)+\cdots+C_h(a_1I)+
Ia_0=0$.{\rm )}
\end{lem}
\begin{proof}
Let $T=C_{h}^{n}+C_{h}^{n-1}a_{n-1}+\cdots+C_ha_1+Ia_0$. We will show that $T$ acts as the
zero endomorphism of $R^{n}_{R}$, and to do so it suffices to show that $Te_i=0$ for all $i$,
where
$\{e_1,\ldots, e_n\}$ is the standard basis for $R^n_R$ (expressed as column vectors). By
construction of $C_h$, $Te_1=0$. Note that for $a\in R$, $(aI)e_i=e_ia$ (whereas this is not
generally true for elements of $R^n_R$). Now,
\begin{equation*}
\begin{split}
Te_i&=C_{h}^{n}e_i+C_{h}^{n-1}(a_{n-1}I)e_i+\cdots+C_h(a_1I)e_i+Ia_0e_i\\
&=C_{h}^{n}e_i+C_{h}^{n-1}e_ia_{n-1}+\cdots+C_he_ia_1+Ie_ia_0.
\end{split}
\end{equation*}
At this point, note that $e_i=C_{h}^{i-1}e_1$, so $Te_i=C_{h}^{i-1}Te_1=0$.
\end{proof}

\begin{lem}\label{lem:9} Let $R$ be a local ring, and suppose that $w_0,w_1\in J(R)$. Write
$h(t)=t^2-(1+w_1)t-w_0$ and consider its companion matrix
\begin{equation*}
C_h=\begin{pmatrix}0&w_0\\
                       1&1+w_1\end{pmatrix}.
\end{equation*}
Then, $C_h$ is strongly clean if and only if $h(t)$ has a left root in $J(R)$ and a left root in
$1+J(R)$.
\end{lem}
\begin{proof}
This is essentially the argument used in \cite[Theorem 3.7.2]{D}. Note that, as matrices,
$C_{h}^{2}-C_h(1+w_1)-Iw_0=0$ by Lemma \ref{lem:companion matrix}. Now, if $C_h$ is a strongly clean
element of
${\mathbb M}_2(R)$, it must be nontrivial, so it acts as a
nontrivial strongly clean endomorphism $\varphi$ of $_RR^2$. So, by Lemma \ref{lem:1},
we can find a decomposition $R^2=P\oplus Q$ where $_RP$ and $_RQ$ are $\varphi$-invariant,
$\varphi$ acts as an automorphism on $P$ and  $1-\varphi$ acts as an isomorphism on $Q$. Since
$\varphi$ is nontrivial strongly clean, $P$ and $Q$ each has
rank $1$. Pick vectors $v_1$ and $v_2$ which are bases of $P$ and $Q$, respectively.
Both $v_1$ and $v_2$ must have at least one coordinate which is a unit.
We can multiply each of $v_1$ and $v_2$ on the left by a unit to assume
that $v_1$ and $v_2$ each have a coordinate which is $1$. Now, for $z
\in \{v_1,v_2\}$, $(z)\varphi$ is in the span of $z$
(since $P$ and $Q$ are $\varphi$-invariant), so $(z)\varphi
=\lambda z$ for some $\lambda$. It is easy to see that the corresponding $\lambda$ for $v_1$ is in
$1+J(R)$ and the other is in $J(R)$. (For instance, see Lemma \ref{lem:2}, or simply
find an explicit $v_1$
and $v_2$ modulo $J(R)$, and lift appropriately to $R$.) Now,
$0=z\big(C_{h}^{2}-C_h(1+w_1)-Iw_0\big)
=\lambda ^2z-\lambda z(1+w_1)-zw_0$. Comparing the component in which $z$ has a $1$,
we have $\lambda^2-\lambda(1+w_1)-w_0=0$. The two $\lambda$ which we have found were
in $J(R)$ and $1+J(R)$, so we have
our left roots of the polynomial $h$ in $J(R)$ and $1+J(R)$.

For the reverse implication, suppose that $\lambda _1\in J(R)$ and
$\lambda_2\in 1+J(R)$ are left roots
of $h$. We will produce a basis of $R^2$ consisting of eigenvectors of $\varphi$.
Consider the row vectors
$v_1=(1,\lambda _1)$ and $v_2=(1,\lambda_2)$ which are easily seen to be
a basis for $_RR^2$ (e.g. one can easily row reduce the corresponding matrix to the identity).
Note that
\begin{equation*}
(v_i)\varphi=\big(\lambda_i, w_0+\lambda_i(1+w_1)\big)=\lambda_i(1,\lambda_i)=\lambda_iv_i.
\end{equation*}
Set $P=Rv_2$ and $Q=Rv_1$. It is clear that $P$ and $Q$ are $\varphi$-invariant, and
that $P\oplus Q=$ $_RR^2$, and that furthermore, $\varphi$ is an isomorphism on $P$,
and $1-\varphi$ is an isomorphism
on $Q$. So $\varphi$ is strongly clean in ${\rm End}(_RR^2)$ by Lemma \ref{lem:1}.
\end{proof}

In \cite[Theorem 3.7.2]{D}, Dorsey proved that for an arbitrary ring $R$, if ${\mathbb M}_n(R)$ ($n\ge 1$) is strongly clean then for each $j\in J(R)$ $t^2-t-j$ has a root in $J(R)$.

For convenience in stating Theorem 7, let
\begin{equation*}
W=\{f\in R[t]: f \,\,\,{\text{is of degree 2, monic, and}}\,\,\,\bar f(0)=\bar f(1)=0\},
\end{equation*}
where $\bar f$ is the image of $f$ in $(R/J)[t]$. Note that $f\in W$ if and only if there are $w_0, w_1\in J(R)$ for which $f(t)=t^2-(1+w_1)t+w_0$.

When doing the second revision of this paper, it came to our attention that, independently,  Bing-jun Li \cite{Li08} has also recently proved the equivalence $(1)\Leftrightarrow (4)$ of Theorem 7.
\begin{thm}\label{thm:11}
The following are equivalent for a local ring $R$:
\begin{enumerate}
\item ${\mathbb M}_2(R)$ is strongly clean.
\item For any $A\in {\mathbb M}_2(R)$, either $A$ is invertible or $I-A$ is invertible or
$A$ is similar to a diagonal matrix.
\item For any $w_0,w_1\in J(R)$, $\begin{pmatrix}0&w_0\\
                       1&1+w_1\end{pmatrix}$ is strongly clean.
\item Every $f\in W$ has a left root in $J(R)$ and a left root in  $1+J(R)$.
\item Every $f\in W$ has a left root in $J(R)$.
\item Every $f\in W$ has a left root in $1+J(R)$.
\item The versions of {\rm (4)} or {\rm (5)} or {\rm (6)} with ``left root'' replaced by ``right root''.
\end{enumerate}
\end{thm}
\begin{proof}
$(1)\Leftrightarrow (2)$. ``$\Rightarrow $'' is by the notice after Lemma 2, and ``$\Leftarrow$'' is
clear, since $R$ is local.

$(1)\Leftrightarrow (3)$. Follows immediately from Lemma 4.

$(3)\Leftrightarrow (4)$.  This follows immediately from Lemma \ref{lem:9}.

The equivalence $(5)\Leftrightarrow (6)$ follows from the fact that $f\in W$ if and only if $g(t)=f(1-t)\in W$. Since $(4)$ is the same as $(5)$ plus $(6)$, it follows that $(4)$, $(5)$ and $(6)$ are equivalent.

Finally, $(1)$ is left-right symmetric in the sense that ${\mathbb M}_2(R)$ is strongly clean if and only if ${\mathbb M}_2(R^{op})$ is strongly clean (note that $R^{op}$ is a local ring). The ``right'' analogues of statements $(4)$-$(6)$ are simply the corresponding ``left'' statements for $R^{op}$, which are equivalent by the equivalence of $(1)$-$(6)$ for the opposite ring $R^{op}$, which is local.
\end{proof}

\medskip
In \cite{BDD05a}, for a commutative local ring $R$, the authors defined the notion of an
${\rm SRC}$ (resp., ${\rm SR}$) factorization of a monic polynomial over $R$,
and proved that ${\mathbb M}_n(R)$ is strongly clean iff
every monic polynomial of degree $n$ over $R$ has an ${\rm SRC}$ factorization.
As an easy corollary of Theorem \ref{thm:11},
there is an analog of this for the $2\times 2$ matrix ring over a local ring.
The next definition extends the notion of an ${\rm SRC}$ (resp., ${\rm SR}$)
factorization from a commutative local ring to a local ring.
We are deliberately not using the term {\rm SRC}, because we do not know whether
the definition is the appropriate generalization of {\rm SRC} for general $n$.

\begin{defn}\label{defn:12} Let $R$ be a local ring. A monic polynomial $f(t)\in R[t]$ is said to
have a $(*)$-factorization if $f(t)=g_0(t)g_1(t)=h_{1}(t)h_{0}(t)$, where $g_0(t), g_1(t)$,
$h_{0}(t)$, $h_{1}(t)\in R[t]$ are monic polynomials such that $g_0(0), g_1(1), h_{0}(0),
h_{1}(1)\in U(R)$. If in addition
$\overline{R}[t]\bar{g_{0}}(t)+\overline{R}[t]\bar {g_{1}}(t)=\overline{R}[t]$
and $\bar{h_{0}}(t)\overline{R}[t]+\bar{h_{1}}(t)\overline{R}[t]=\overline{R}[t]$ hold, then
$f(t)$ is said to have a $(**)$-factorization.
\end{defn}

It is interesting to compare the next result with \cite[Corollary 15, Proposition 17]{BDD05a}.

\begin{cor}\label{cor:13}
The following are equivalent for a local ring $R$:
\begin{enumerate}
\item ${\mathbb M}_{2}(R)$ is strongly clean.
\item Every companion matrix in ${\mathbb M}_2(R)$ is strongly clean.
\item Every monic quadratic polynomial over $R$ has a $(*)$-factorization.
\item Every monic quadratic polynomial over $R$ has a $(**)$-factorization.
\end{enumerate}
\end{cor}
\begin{proof} $(1)\Leftrightarrow (2)$. This holds by the equivalence of `$(1)\Leftrightarrow (4)$'
of Theorem \ref{thm:11}.

$(1)\Rightarrow (4)$. Let $f(t)=t^2+at+b\in R[t]$. If $f(0)\in U(R)$ or $f(1)\in U(R)$, then
\begin{equation*}
f(t)=
\left\{\begin{array}{ll}
1\cdot f(t)=f(t)\cdot 1, &\text{if }f(1)\in U(R);\\
f(t)\cdot 1=1\cdot f(t), &\text{if } f(0)\in
U(R)
\end{array}
\right.
\end{equation*}
is a $(**)$-factorization.  So assume that $f(0), f(1)\in J(R)$.
Then $b\in J(R)$ and $-a=1+(b-f(1))\in 1+J(R)$. By Theorem \ref{thm:11},
$f(t)$ has a left root $t_0\in J(R)$ and a left root $t_1\in 1+J(R)$.
Thus, $f(t)=(t-t_1)(t+a+t_1)=(t-t_0)(t+a+t_0)$ is clearly a
$(**)$-factorization.

$(4)\Rightarrow (3)$. It is obvious.

$(3)\Rightarrow (1)$. For $w_0, w_1\in J(R)$, $f(t)=t^{2}-(1+w_{1})t-w_{0}$
has a $(*)$-factorization. This clearly shows that $f(t)$ has a left root in $J(R)$ and
a left root in $1+J(R)$ by $(3)$.
Hence $(1)$ holds by Theorem \ref{thm:11}.
\end{proof}

\bigskip
\section{Applications and examples}
Conditions (4)-(7) of Theorem \ref{thm:11} are easily verifiable criteria for a $2\times 2$ matrix
ring over a local ring to be strongly clean. We use them here to give new families of strongly clean
rings.

\medskip
For an ideal $I$ of a ring $R$, let $\overline R=R/I$ and write $\bar r=r+I$ for $r\in R$. Further,
for
$f(t)=\sum a_it^i\in R[t]$, we write $\bar f(t)=\sum \bar a_it^i\in \overline R[t]$.
\begin{defn}\label{def:14}\cite{MA06}
A local ring $R$ (may not be commutative)  with ${\overline R}:=R/J(R)$
being a field is called {\it Henselian} if $R[t]$ satisfies Hensel's
lemma: for any monic polynomial $f(t)\in R[t]$, if
$\overline{f}(t)=\alpha (t)\beta(t)$ with
$\alpha(t)$, $\beta(t) \in {\overline R}[t]$ monic and
coprime, then there exist unique monic polynomials $g(t)$ and $h(t)$
in $R[t]$ such that $f(t)=g(t)h(t),\,\, \overline{g}(t)=\alpha(t)$,  and
$\overline{h}(t)=\beta(t)$.
\end{defn}

The authors of \cite{BDD05a} proved that matrix rings of arbitrary size over a commutative Henselian
ring
are all strongly clean (\cite[Example 22]{BDD05a}). With Theorem \ref{thm:11} in hand, the same type
of proof yields the analogous result for $2\times 2$ matrices over arbitrary Henselian rings.

\begin{prop}\label{prop:15}
Let $R$ be a Henselian ring. Then ${\mathbb M}_{2}(R)$ is strongly clean.
\end{prop}
\begin{proof}Let $w_0,w_1\in J(R)$ and let $f(t)=t^2-(1+w_1)t-w_0$. Then $\bar f(t)=t^2-t=
t(t-1)\in \overline R[t]$. By hypothesis, there exist monic polynomials $t-a, t-b\in R[t]$ such that
$f(t)=(t-a)(t-b)$ and $t-\bar a=t$ and $t-\bar b=t-1$. It follows that $a\in J(R)$ is a left root of
$f(t)$.  Hence ${\mathbb M}_{2}(R)$ is strongly clean by Theorem \ref{thm:11}.
\end{proof}
A Henselian ring that is not commutative can be found in  \cite[Example 16]{MA06}.
In order to give another family of strongly clean matrix rings, we need a new notion.

\medskip
Following \cite{BDD05b}, a local ring $R$ is called {\it bleached} if, for all $j\in J(R)$ and
$u\in U(R)$, the additive abelian group endomorphisms ${\it l}_u-{\it r}_j: R\rightarrow R$
($x\mapsto ux-xj$) and ${\it l}_j-{\it r}_u: R\rightarrow R$ ($x\mapsto jx-xu$) are surjective.
By \cite[Example 13]{BDD05b}, some examples of bleached local rings include: commutative local
rings, division rings, local rings $R$ with $J(R)$ nil, local rings $R$ for which some power of each
element of $J(R)$ is central in $R$, local rings $R$ for which some power of each element of $U(R)$
is central in $R$, power series rings over bleached local rings, and skew power series rings
$R[[x; \sigma]]$ of a bleached local ring $R$ with $\sigma$ an automorphism of $R$.

\begin{defn}\label{def:17}
 A local ring $R$ is called {\it weakly bleached} if, for all $j_1, j_2\in J(R)$, the additive
abelian group endomorphisms ${\it l}_{1+j_1}-{\it r}_{j_2}$ and ${\it l}_{j_2}-{\it r}_{1+j_1}$
are surjective.
\end{defn}
By Nicholson \cite[Example 2]{N99} (also see \cite[Theorem 18]{BDD05b}), a local ring $R$ is weakly
bleached  iff the $2\times 2$ upper triangular matrix ring ${\mathbb T}_2(R)$ is strongly clean.
There exist examples, however, of local rings which are not weakly bleached
(e.g. \cite[Example 45]{BDD05b}). Bleached rings are clearly weakly bleached, but the converse is
not true by \cite[Example 38]{BDD05b} together with \cite[Theorem 30]{BDD05b}.

\medskip
The next result was known in the commutative case when $\sigma =1_R$ (see \cite[Theorem 9]{CYZ06}).
\begin{thm}\label{thm:19}
Let $R$ be a weakly bleached local ring and let   $\sigma:R\rightarrow R$ be an endomorphism with
$\sigma(J(R))\subseteq J(R)$. Then the following are equivalent for $n\ge 1$:
\begin{enumerate}
\item ${\mathbb M}_2(R)$ is strongly clean.
\item ${\mathbb M}_2(R[[x;\sigma]])$ is strongly clean.
\item ${\mathbb M}_2(R[x;\sigma]/(x^n))$ is strongly clean.
\end{enumerate}
\end{thm}
\begin{proof}
$(2)\Rightarrow (3)\Rightarrow (1)$. This follows because the image of a strongly clean ring is
again strongly clean.

$(1)\Rightarrow (2)$. Let $S=R[[x;\sigma]]$. Note that $J(S)=J(R)+Sx$. By Theorem \ref{thm:11}, it
suffices to show that,
for any $w_0, w_1\in J(S)$, $t^2-(1+w_1)t-w_0$ has a left root in $J(S)$.
Write
\begin{equation*}
\begin{split}
w_1&=b_0+b_1x+\cdots,\\
w_0&=c_0+c_1x+\cdots,\\
t&=t_0+t_1x+\cdots,
\end{split}
\end{equation*}
where $b_0, c_0\in J(R)$. Then $t^2-t(1+w_1)-w_0=0$ $\Leftrightarrow$
\begin{equation*}
\begin{split}
\begin{cases}
t_{0}^{2}-t_0(1+b_0)-c_0&=0\,\hspace{6.4cm} (P_0)\\
t_k[1-\sigma^k(t_0)+\sigma^k(b_0)]-t_0t_k&=[t_1\sigma(t_{k-1})+\cdots+
t_{k-1}\sigma^{k-1}(t_1)]\\
&\quad  -[t_0b_k+\cdots+t_{k-1}\sigma^{k-1}(b_1)]-c_k\,\,\,\,\,\hspace{1.3cm}(P_k)
\end{cases}
\end{split}
\end{equation*}
for $k=1,2,\ldots$. By Theorem \ref{thm:11}, $t^{2}-(1+b_0)t-c_0$ has a left root $t_0\in J(R)$.
Thus, $1-\sigma^k(t_0)+\sigma^k(b_0)\in 1+J(R)$, so $(P_k)$ is solvable for $t_k$ (because $R$ is
weakly bleached) for $k=1,2,\ldots$. Thus, $\sum _it_ix^i\in J(S)$ is a left root of
$t^2-(1+w_1)t-w_0$.
The proof is complete.
\end{proof}

It is unknown if the commutative Henselian rings are exactly those commutative local rings over
which the matrix rings are strongly clean (see \cite[Problem 23]{BDD05a}). The next example
gives a (noncommutative) local ring $R$ that is not Henselian such that ${\mathbb M}_2(R)$ is
strongly clean.

\begin{exa}\label{exa:21}
Let $D$ be a division ring and $\sigma$ an endomorphism of $D$. Then ${\mathbb M}_2(D[[x;\sigma]])$
is strongly clean by Theorem \ref{thm:19}. If, in particular, $D={\mathbb C}$ and $\sigma$ is the
complex conjugation, then $D[[x;\sigma]]$ is not Henselian by \cite[Example 17]{MA06}.
\end{exa}

The next corollary follows by Proposition \ref{prop:15} and Theorem \ref{thm:19}.
\begin{cor}\label{cor:22}
If $R$ is a weakly bleached Henselian ring and $\sigma$ is an endomorphism of $R$ with
$\sigma(J(R))\subseteq J(R)$, then ${\mathbb M}_2(R[[x; \sigma]])$ and
${\mathbb M}_2\big(R[x; \sigma]/(x^n)\big)$ are strongly clean.
\end{cor}

\section{Strongly $\pi$-regular matrices}

In this section, we characterize the local rings $R$ for which ${\mathbb M}_2(R)$ is strongly
$\pi$-regular.
This topic is included here mainly because the techniques involved are very similar to those in
previous sections.
\begin{lem} \label{lem:23}\cite{N99}
Let $M_R$ be a module. The following are equivalent for $\varphi\in {\rm End}(M_R)$:
\begin{enumerate}
\item $\varphi$ is strongly $\pi$-regular in ${\rm End}(M_R)$.
\item There is a decomposition $M=P\oplus Q$ where $P$ and $Q$ are $\varphi$-invariant, and
$\varphi |_{P}$ is an isomorphism and $\varphi |_Q$ is nilpotent.
\end{enumerate}
\end{lem}

Units and nilpotent elements of a ring are clearly strongly $\pi$-regular elements. A strongly
$\pi$-regular element $a\in R$ is called a {\it trivial strongly $\pi$-regular element} if $a$ is a
unit or nilpotent, and is called {\it non-trivial} otherwise. Because of Lemma \ref{lem:23}, the
same proof of Lemma \ref{lem:2} works for the next lemma, which is a translation of the
decomposition in Lemma \ref{lem:23} to matrices. The hypothesis here is based on following the
approach of \cite{BDD05a}.

\begin{lem}\label{lem:24}
Let $R$ be a projective-free ring. Then
$A\in {\mathbb M}_n(R)$ is a non-trivial strongly $\pi$-regular matrix iff $A$ is similar to
$\begin{pmatrix}T_0&0\\
                       0&T_1\end{pmatrix}$, where $T_0$ is an invertible matrix and $T_1$ is a
nilpotent matrix.
\end{lem}

\begin{cor}\label{cor:25}
Let $R$ be a local ring. Then $A\in {\mathbb M}_2(R)$ is a non-trivial strongly $\pi$-regular
matrix iff
$A$ is similar to $\begin{pmatrix}t_0&0\\
                       0&t_1\end{pmatrix}$, where $t_0\in U(R)$ and $t_1\in R$ is nilpotent.
\end{cor}

As pointed out in \cite{BDD05a}, it follows from the results of the literature that for any
commutative ring $R$, ${\mathbb M}_n(R)$ is strongly $\pi$-regular iff so is $R$ and that, for a
commutative local ring $R$, ${\mathbb M}_n(R)$ is strongly $\pi$-regular iff so is $R$ iff $J(R)$
is nil. By \cite{BDD05a}, there exists a commutative local ring $R$ such that ${\mathbb M}_2(R)$
is strongly clean, but not strongly $\pi$-regular. Below, we characterize the local rings $R$
for which ${\mathbb M}_2(R)$ is strongly $\pi$-regular.

\begin{lem}\label{lem:27}
Let $A\in {\mathbb M}_2(R)$ where $R$ is a local ring.
If $A\notin {\mathbb M}_2(J(R))\cup {\rm GL}_2(R)$, then $A$ is similar to
$\begin{pmatrix}0&w\\
                       1&r\end{pmatrix}$ where $w\in J(R)$ and $r\in R$.
\end{lem}
\begin{proof}
Let $\varphi\in {\rm End}(R^2_R)$, where $\varphi \notin J({\rm End}(R^2_R))$ and
$\varphi$ is not invertible. We show that there exists a cyclic basis $\{x, \varphi(x)\}$ of
$R^2_R$;
with respect to this basis, $\varphi$ corresponds to
$\begin{pmatrix}0&w\\
                       1&r\end{pmatrix}$ for some $w\in J(R)$ and $r\in R$.
Because  $\overline{\varphi}$ is is not a unit (since units lift modulo the radical),  ${\rm ker}(\overline \varphi)\not= 0$, but also  ${\rm Im}(\overline \varphi)\not= 0$, since  $\overline{\varphi}\not= 0$
(since $\varphi \notin J({\rm End}(R^2_R))$).  In particular, by the rank-nullity theorem, both
 ${\rm ker}(\overline \varphi)$
and ${\rm Im}(\overline \varphi)$ are $1$-dimensional. It follows that
${\rm Im}(\overline \varphi)\cup {\rm ker}(\overline \varphi)\not= {\overline R}^2$ (since a vector space is never
the union of two proper subspaces). Take $v$ outside of the union, and look at $\{v, \overline {\varphi}(v)\}$.
Note that $\overline {\varphi}(v)\not= 0$. And since $v$ is not in ${\rm Im}(\overline \varphi)$, $\{v, \overline {\varphi}(v)\}$ is independent.
Now, lift $v$ to $x$ in $R^2$. Then, modulo $JR^2$, $\{x, \varphi(x)\}$ is
$\{v, \overline {\varphi}(v)\}$, which is a basis for ${\overline R}^2_{\overline R}$.
So $x$ and $\varphi(x)$ span $R^2_R$ by Nakayama's Lemma. Moreover,
$\{x, \varphi(x)\}$ is a basis for $R^2_R$ since every local ring is stably finite.
Write $\varphi^2(x)=xa+
\varphi(x)b$. Reducing modulo $JR^2$, this equation becomes
${\overline \varphi}^2(v)=v\bar a+{\overline \varphi}(v)\bar b$. Since
$v$ is not in ${\rm Im}(\overline \varphi)$,
we conclude that $\bar a=0$. This is, $a\in J(R)$. The matrix
representation of $\varphi$ with respect to the basis $\{x, \varphi(x)\}$
is $\begin{pmatrix}0&a\\
                       1&b\end{pmatrix}$,  with $a\in J(R)$, as desired.
\end{proof}

\begin{lem}\label{lem:28}
Let $R$ be a local ring, and suppose that $u\in U(R)$ and $w\in J(R)$. Write $h(t)=t^2-ut-w$ and
consider
its companion matrix
\begin{equation*}
C_h=\begin{pmatrix}0&w\\
                       1&u\end{pmatrix}.
\end{equation*}
Then, $C_h$ is strongly $\pi$-regular if and only if $h(t)$ has  two left roots, one
in $U(R)$ and one which is nilpotent.
\end{lem}
\begin{proof}
The proof of Lemma 6 proves this statement, appealing to Lemma 16 instead of Lemma 1, and making
the resulting obvious changes.
\end{proof}

\begin{thm}\label{thm:29} The following are equivalent for a local ring $R$:
\begin{enumerate}
\item ${\mathbb M}_2(R)$ is strongly $\pi$-regular.
\item ${\mathbb M}_2(J(R))$ is nil and, for any $u\in U(R)$ and $w\in J(R)$, $t^2-ut-w$ has two left
roots, one in $U(R)$ and one in $J(R)$.
\item ${\mathbb M}_2(J(R))$ is nil and, for any $u\in U(R)$ and $w\in J(R)$, $t^2-ut-w$ has two
right roots, one in $U(R)$ and one in $J(R)$.
\end{enumerate}
\end{thm}
\begin{proof}
$(1)\Rightarrow (2)$. $(1)$ clearly implies that ${\mathbb M}_2(J(R))$ is nil. For $u\in U(R)$
and $w\in J(R)$,
let $A=\begin{pmatrix}0&w\\
                       1&u\end{pmatrix}$.  By $(1)$, $A$ is strongly $\pi$-regular. Hence, by
Lemma \ref{lem:28},
                       $t^2-ut-w$ has two left roots, one in $U(R)$ and one which is nilpotent.  So
$(2)$ holds.

$(2)\Rightarrow (1)$. Let $A\in {\mathbb M}_2(R)$. We want to show that $A$ is strongly
$\pi$-regular. Because ${\mathbb M}_2(J(R))$ is nil, and every nilpotent element of a ring is strongly $\pi$-regular,
we may assume that $A\notin {\mathbb M}_2(J(R))$ and
$A\notin {\rm GL}_2(R)$. Thus, by Lemma \ref{lem:27}, we may assume that
$A=\begin{pmatrix}0&w\\
                       1&u\end{pmatrix}$ where $u\in R$ and $w\in J(R)$; moreover, we may further
assume that $u\in U(R)$, for otherwise $A^2\in {\mathbb M}_2(J(R))$,
so $A$ is nilpotent. By $(2)$,  $t^2-ut-w=0$ has two left roots,
one in $U(R)$ and one in $J(R)$ (which is nilpotent). Thus, by Lemma \ref{lem:28}, $A$ is strongly
$\pi$-regular.

$(1)\Leftrightarrow (3)$. Similar to the proof of  $(1)\Leftrightarrow (2)$, or alternatively, appeal to the
opposite ring, as in the proof of Theorem 7.
\end{proof}

As mentioned before, for a commutative local ring $R$, ${\mathbb M}_2(R)$ is
strongly $\pi$-regular iff $J(R)$ is nil. As a contrast of this, there exists a local ring
$R$ with $J(R)$ locally nilpotent (thus, ${\mathbb M}_2(J(R))$ is nil), but ${\mathbb M}_2(R)$ is
not strongly $\pi$-regular by \cite{CR98}.
For a left perfect ring $R$, ${\mathbb M}_n(R)$ is again left perfect, so it is
strongly $\pi$-regular. It is worth noting that there exists a non commutative local ring $R$ that is
not one-sided perfect such that ${\mathbb M}_n(R)$ (for each $n\ge 1$) is strongly $\pi$-regular.

\bigskip

\section*{Acknowledgments}
The authors' sincere thanks go to the referee for his insights and numerous valuable comments in
two extensive reports, which have largely improved the presentation of the paper. In particular, the conceptual proofs of Lemmas 4,6, 19 and 20 were suggested to the authors by the referee to replace their matrix-theoretic arguments.
The research was supported by NSERC of Canada (Grant OGP0194196). The first named author was
grateful to the A.G. Hatcher Scholarship and the financial support received from Memorial University of Newfoundland and the Atlantic Algebra Center at MUN respectively.

\end{document}